\documentclass[leqno]{amsart}
\usepackage{amsmath,amssymb,amsthm,mathtools,dsfont}

\DeclarePairedDelimiter\norm{\lvert}{\rvert}
\DeclarePairedDelimiter\inner{\langle}{\rangle}
\newcommand{\Char}{\mathrel{{\rm char}}} 

\newcommand{\divides}[1][]{
\def\DDD{
\mathrel #1\vert
}
\def\DTT{
\mathrel #1\vert
}
\def\DSS{
\,\mathrel \vert \,
}
\def\DSC{
\mathrel \vert 
}
\mathchoice{\DDD}{\DTT}{\DSS}{\DSC}
}
\usepackage{bm}
\usepackage[colorlinks=true,linkcolor=blue,citecolor=blue]{hyperref}
\usepackage{enumitem}
\usepackage{csquotes}

\newcounter{intro}

		\newtheorem{introthm}[intro]{Theorem}

		\newtheorem{thm}[equation]{Theorem}
		\newtheorem{lem}[equation]{Lemma}
		\newtheorem{cor}[equation]{Corollary}
		\newtheorem{prop}[equation]{Proposition}
		\newtheorem{conj}[equation]{Conjecture}
\theoremstyle{remark}
		\newtheorem{rem}[equation]{Remark}
\theoremstyle{definition}
		
		\newtheorem{exam}[equation]{Example}

\title{An Analog of Nilpotence Arising from Supercharacter Theory}
\author{Shawn T. Burkett}
\address{Department of Mathematical Sciences, Kent State University, Kent,
Ohio 44240, U.S.A.} \email{sburket1@kent.edu}
\date{\today}
\begin{document}
\maketitle
\begin{abstract}
The goal of this paper is to generalize several group theoretic concepts such as the center and commutator subgroup, central series, and ultimately nilpotence to a supercharacter theoretic setting, and to use these concepts to show that there can be a strong connection between the structure of a group and the structure of its supercharacter theories. We then use these concepts to show that the upper and lower annihilator series of $J$ can be described in terms of certain central series for the algebra group $G=1+J$ defined by $\mathsf{S}$, when $\mathsf{S}$ is the algebra group supercharacter theory defined by Diaconis--Isaacs.
\end{abstract}
\section{Introduction}In \cite{ID07}, Diaconis and Isaacs defined supercharacter theories of a finite group $G$, objects which approximate the irreducible characters of $G$ while preserving much of the duality between the sets of irreducible characters and conjugacy classes. Their work generalizes that of Andr\'{e} \cite{CA95,CA02} and Yan \cite{NY01}, which aimed to build a suitable set of characters for the groups $\mathrm{UT}_n(\mathbb{F}_q)$ of unipotent upper triangular matrices over finite fields. The problem of classifying the irreducible characters of these groups is well known to be \enquote{wild}; however, their work produced mutually orthogonal characters constant on unions of conjugacy classes, the values of which can be computed using combinatorial formulae. Using supercharacter theory, Diaconis and Isaacs generalized these results to a larger class of groups, called $\mathbb{F}_q$-algebra groups, groups that arise from finite dimensional, associative, nilpotent $\mathbb{F}_q$-algebras. 

Since this time, supercharacter theory has been used to study very specific families of groups, very often related to algebra groups. Examples include the work of Diaconis, Marberg, and Thiem on supercharacter theories of pattern groups \cite{NTPD09,NTEM09}, the work of Andrews, Diaconis--Isaacs, and Marberg on the supercharacter theories of algebra groups  \cite{CA08} and, more generally, the work of Andr\'{e} and Nicol\'{a}s on supercharacter theories of adjoint groups of radical rings \cite{CA08}. Outside of the realm of algebra groups, Hendrickson classified the supercharacter theories of cyclic groups of prime power order \cite{AH08}, and Lewis--Wynn gave several results describing the structure of supercharacter theories of semi-extraspecial groups and groups possessing a Camina pair \cite{camina}. However, the possibility for using supercharacter theory as a tool for studying arbitrary finite groups seems to be relatively unexplored. In this paper, we begin a framework for studying finite groups via their supercharacter theories, by generalizing several classical concepts to this coarser setting. 

As defined in \cite{ID07}, a supercharacter theory $\mathsf{S}$ of a group $G$ is a pair $(\mathcal{X}_{\mathsf{S}},\mathcal{K}_{\mathsf{S}})$, where $\mathcal{X}_{\mathsf{S}}$ is a partition of the set $\mathrm{Irr}(G)$ of irreducible characters of $G$, and $\mathcal{K}_{\mathsf{S}}$ is a partition of $G$ satisfying
\begin{itemize}\openup 5pt
\item $\{1\}\in\mathcal{K}_{\mathsf{S}}$;
\item $\norm{\mathcal{X}_{\mathsf{S}}}=\norm{\mathcal{K}_{\mathsf{S}}}$;
\item for each $X\in\mathcal{X}_{\mathsf{S}}$, there is a character $\xi^{}_X$ whose constituents lie in $X$ which is constant on the parts of $\mathcal{K}_{\mathsf{S}}$.
\end{itemize}
It is shown in \cite{ID07} that each part of $\mathcal{K}_{\mathsf{S}}$ is a union of conjugacy classes, $\{\mathds{1}\}\in\mathcal{X}_{\mathsf{S}}$, and each character $\xi_X$ is (up to a scalar) the character $\sigma_X=\sum_{\psi\in X}\psi(1)\psi$. 

We will write $\mathrm{Cl}(\mathsf{S})$ instead of $\mathcal{K}_{\mathsf{S}}$, and call its elements $\mathsf{S}$-{\bf classes}. We often refer to $\mathsf{S}$-classes as {\bf superclasses} if there is no ambiguity. For $g\in G$, we will let $\mathrm{cl}_{\mathsf{S}}(g)$ denote the $\mathsf{S}$-class containing $G$. Similarly, the set $\{\sigma_X:X\in\mathcal{X}_{\mathsf{S}}\}$ will be denoted $\mathrm{Ch}(\mathsf{S})$, and its elements will be called $\mathsf{S}$-{\bf characters} or {\bf supercharacters}. For any character $\chi$ of $G$, we will let $\mathrm{Irr}(\chi)$ denote the set of irreducible constituents of $\chi$. We also let $\mathrm{SCT}(G)$ denote the set of all supercharacter theories of $G$.

Let $\preccurlyeq$ be the refinement order on partitions, and let $\vee$ be the mutual coarsening operation on partitions. Then
\[(\mathcal{X}\vee\mathcal{Y},\mathcal{K}\vee\mathcal{L})\]
is a supercharacter theory of $G$ for any supercharacter theories $\mathsf{S}=(\mathcal{X},\mathcal{K})$ and $\mathsf{T}=(\mathcal{Y},\mathcal{L})$ of $G$ \cite[Proposition 3.3]{AH12}. We write $\mathsf{S}\vee\mathsf{T}$ to denote this supercharacter theory. From \cite[Corollary 3.4]{AH12}, we have that every $\mathsf{S}$-class is  union of $\mathsf{T}$-classes if and only if every $\mathsf{S}$-character is a sum of $\mathsf{T}$-characters. This fact allows one to extend $\preccurlyeq$ to a partial order on $\mathrm{SCT}(G)$ by defining $\mathsf{S}\preccurlyeq\mathsf{T}$ if $\mathrm{Cl}(\mathsf{S})\preccurlyeq\mathrm{Cl}(\mathsf{T})$. In this event, we say that $\mathsf{T}$ is finer than $\mathsf{S}$, and that $\mathsf{S}$ is coarser than $\mathsf{T}$. Also, $\mathrm{SCT}(G)$ has a minimal element $\mathsf{m}(G)$, where the $\mathsf{m}(G)$-classes are just the usual conjugacy classes of $G$, and the $\mathsf{m}(G)$-characters are the characters $\chi(1)\chi$ for $\chi\in\mathrm{Irr}(G)$. In particular, $\mathrm{SCT}(G)$ is a lattice, and we let $\wedge$ denote its meet operation. The maximal element of $\mathrm{SCT}(G)$ is the supercharacter theory with  superclasses $\{1\}$ and $G\setminus\{1\}$, and supercharacters $\mathds{1}$ and $\rho_G-\mathds{1}$, where $\rho_G$ denotes the regular character of $G$. We denote this supercharacter theory by $\mathsf{M}(G)$.

A subgroup $N\le G$ is called $\mathsf{S}$-{\bf normal} if $N$ is the union of some $\mathsf{S}$-classes of $G$. Equivalently $N$ is $\mathsf{S}$-normal if and only if $N$ arises as the kernel of some sum of $\mathsf{S}$-characters \cite[Proposition 2.1]{EM11}. We will use the notation $N\lhd_{\mathsf{S}}G$ to denote that $N$ is an $\mathsf{S}$-normal subgroup of $G$. Hendrickson \cite[Proposition 6.4]{AH12} showed  that there is an induced supercharacter theory $\mathsf{S}_N$ on $N$ and $\mathsf{S}^{G/N}$ on $G/N$ whenever $N\lhd_{\mathsf{S}}G$, where
\begin{align*}
\mathrm{Ch}(\mathsf{S}_N)&=\bigl\{\sigma_{\mathrm{Irr}(\mathrm{Res}^G_N(\chi))}:\chi\in\mathrm{Ch}(\mathsf{S})\bigr\}\\
\mathrm{Cl}(\mathsf{S}_N)&=\bigl\{K\in\mathrm{Cl}(\mathsf{S}):K\subseteq N\bigr\}
\end{align*}
and
\begin{align*}
\mathrm{Ch}(\mathsf{S}^{G/N})&=\bigl\{\mathrm{Def}^G_{G/N}(\chi):\chi\in\mathrm{Ch}(\mathsf{S})\bigr\}\\
\mathrm{Cl}(\mathsf{S}^{G/N})&=\bigl\{\pi(K):K\in\mathrm{Cl}(\mathsf{S})\bigr\},
\end{align*}
where $\pi:G\to G/N$ is the canonical projection. 

\begin{introthm}\label{degreediv}
Let $\mathsf{S}$ be a supercharacter theory of $G$, and let $N$ be $\mathsf{S}$-normal. The following hold.
\begin{enumerate}[label={\bf(\arabic*)}]\openup5pt
\item For each $\chi\in\mathrm{Ch}(\mathsf{S})$ and $\psi\in\mathrm{Ch}(\mathsf{S}_N)$ satisfying $\inner{\psi,\mathrm{Res}^G_N(\chi)}>0$, $\psi(1)$ divides $\chi(1)$. In particular, $\mathrm{Res}^G_N(\chi)$ is a positive integer multiple of an $\mathsf{S}_N$ character.
\item For each $g\in G$, $\norm{\mathrm{cl}_{\mathsf{S}^{G/N}}(gN)}$ divides $\norm{\mathrm{cl}_{\mathsf{S}}(g)}$.
\end{enumerate}
\end{introthm}

Given a supercharacter theory $\mathsf{S}$ of $G$, we define analogs of the center and commutator subgroup of $G$ as follows. The analog $\mathbf{Z}(\mathsf{S})$ of the center consists of the elements of $G$ that lie in $\mathsf{S}$-classes of size one, and the fact that this set forms a group is well known. For $H\le G$, we let $[H,\mathsf{S}]$ be the subgroup generated by all elements of the form $g^{-1}k$, where $g\in H$ and $k$ lies in the same $\mathsf{S}$-class as $g$. We show that $H$ is $\mathsf{S}$-normal if and only if $[H,\mathsf{S}]\le H$, and in this event $[H,\mathsf{S}]$ is $\mathsf{S}$-normal as well.

We use these subgroups to give a new type of central series in the following way. We say that a series 
\[G=N_1\ge N_2\ge\dotsb\ge N_{r+1}=1\]
of $\mathsf{S}$-normal subgroups is an $\mathsf{S}$-{\bf central series} if
\[N_i/N_{i+1}\le\mathbf{Z}(\mathsf{S}^{G/N_{i+1}})\]
for all $1\le i\le r$. Observe that the name is appropriate since $\mathbf{Z}(\mathsf{S}^{G/N_{i+1}})$ is a central subgroup of $G/N_{i+1}$. We can also define direct analogs of the lower and upper central series of $G$. We define the {\bf lower} $\mathsf{S}$-central series by \[\gamma_i(\mathsf{S})=[\gamma_{i-1}(\mathsf{S}),\mathsf{S}],\] where $\gamma_1(\mathsf{S})=G$ and the {\bf upper} $\mathsf{S}$-central series by $\zeta_0(\mathsf{S})=1$ and \[\zeta_i(\mathsf{S})/\zeta_{i-1}(\mathsf{S})=\mathbf{Z}(\mathsf{S}^{G/\zeta_{i-1}(\mathsf{S})}).\] When they terminate appropriately, these are in fact $\mathsf{S}$-central series.

In the event that $G$ has an $\mathsf{S}$-central series, we will say that $G$ is $\mathsf{S}$-{\bf nilpotent}. Such groups are necessarily nilpotent, but being $\mathsf{S}$-nilpotent for some supercharacter theory $\mathsf{S}$ is more restrictive. 
\begin{introthm}\label{cc}
Let $\mathsf{S}$ be a supercharacter theory of $G$ for which $G$ is $\mathsf{S}$-nilpotent. Then $\chi(1)$  divides $\norm{G:\ker(\chi)}$ for each $\chi\in\mathrm{Ch}(\mathsf{S})$, and $\norm{\mathrm{cl}_{\mathsf{S}}(g)}$ divides $\norm{G}$ for each $g\in G$.
\end{introthm}

We remark here that it is not known if a converse of Theorem~\ref{cc} holds in general, but we will discuss this possibility more in Section~\ref{nilpsection}.  

Of particular interest are $p$-groups, as $\mathsf{S}$-nilpotence has a particularly nice characterization in this case. Specifically, a converse of Theorem~\ref{cc} holds in this case.
\begin{introthm}\label{intropgp}
Let $G$ be a $p$-group, and let $\mathsf{S}$ be a supercharacter theory of $G$. The following are equivalent:
\begin{enumerate}[label = {\bf(\arabic*)}]\openup 5pt
\item the group $G$ is $\mathsf{S}$-nilpotent;
\item the degree of $\chi$ divides $\norm{G}$ for every $\chi\in\mathrm{Ch}(\mathsf{S})$;
\item the size of $\mathrm{cl}_{\mathsf{S}}(g)$ divides $\norm{G}$ for every $g\in G$.
\end{enumerate}
\end{introthm}

In particular, if $G$ is an $\mathbb{F}_q$-algebra group and $\mathsf{S}$ is the {\it double orbit} supercharacter theory of Diaconis--Issacs, then it follows from \cite[Corollary 3.2]{ID07} that $G$ is $\mathsf{S}$-nilpotent. We will show that if $G=1+J$ is an $\mathbb{F}_q$-algebra group, then the terms in the lower and upper $\mathsf{S}$-central series are ideal subgroups of $G$, that is, subgroups of the form $1+I$, where $I$ is an ideal of $J$. Moreover, we show that the associated ideals of the lower and upper $\mathsf{S}$-central series are the terms of the lower and upper annihilator series of $J$, respectively. 

\begin{introthm}\label{idealsubs}
Let $\mathsf{S}$ be the double orbit supercharacter theory for the algebra group $G$, and assume $\norm{G}=q^n$. For each $0\le m\le n$, $G$ has an $\mathsf{S}$-normal ideal subgroup of order $q^m$.
\end{introthm}

This work was a part of the author's PhD thesis at the University of Colorado Boulder under the supervision of Nat Thiem. I would to thank Dr. Thiem for his numerous helpful conversations and suggestions.
\section{$\mathsf{S}$-normality}\label{snormality}

In this section, we give an overview of supernormality, as well as the supercharacter theories $\mathsf{S}_N$ and $\mathsf{S}^{G/N}$ defined by Hendrickson (see \cite{AH12}). 

Let $N$ be $\mathsf{S}$-normal. We will set the following notation. Let $\mathsf{S}$ be a supercharacter theory of $G$, and let $N$ be $\mathsf{S}$-normal. We will let 
\begin{align*}
\mathrm{Ch}(\mathsf{S}\mid N)&=\{\chi\in\mathrm{Ch}(\mathsf{S}): N\nsubseteq\ker(\chi)\}\\
\shortintertext{and}
\mathrm{Ch}(\mathsf{S}/N)&=\{\chi\in\mathrm{Ch}(\mathsf{S}): N\subseteq\ker(\chi)\}.
\end{align*}
For a character $\chi$ of $G$ afforded by a $G$-module $V$, we let $\mathrm{Def}^G_{G/N}(\chi)$ denote the character afforded by the $G/N$-module consisting of the $N$-fixed points of $V$. In the event that $N\le\ker(\chi)$, $\mathrm{Def}^G_{G/N}(\chi)(gN)=\chi(g)$, and we make no distinction between the two. That is, we typically identity characters of $G$ that contain $N$ in their kernels with characters of $G/N$. 

Let $H$ be a normal subgroups of $G$, and let $N$ be a normal subgroup of $G$ contained in $H$. Since $\mathrm{Def}^H_{H/N}\mathrm{Res}^G_H(\chi)=\mathrm{Res}^{G/N}_{H/N}\mathrm{Def}^G_{G/N}(\chi)$ for any $\chi\in\mathrm{Irr}(G/N)$, it follows that 
\[(\mathsf{S}_N)^{H/N}=(\mathsf{S}^{G/N})_{H/N}\]
for any $\mathsf{S}$-normal subgroups $N\le H$. Consequently, we shall henceforth write $\mathsf{S}_{H/N}$ to denote the supercharacter theory induced by $\mathsf{S}$ on the quotient $H/N$.

As with normal subgroups, $\mathsf{S}$-normal subgroups respect quotient structures. 
\begin{lem}\label{latticeisothm}
Let $N$ be $\mathsf{S}$-normal in $G$, and let $M\le G$ contain $N$. Then $M$ is $\mathsf{S}$-normal if and only if $M/N$ is $\mathsf{S}_{G/N}$-normal
\end{lem}
\begin{proof}
Let $\pi:G\to G/N$ be the canonical homomorphism. Then $\bigcup_{m\in M}\pi(\mathrm{cl}_{\mathsf{S}}(m))$ contains no elements outside of $M/N$ if and only if $\bigcup_{m\in M}\mathrm{cl}_{\mathsf{S}}(m)$ contains no elements outside of $M$. 
\end{proof}
We conclude this section with a proof of Theorem~\ref{degreediv}. We will need the following result, which appears as \cite[Lemma 4.2]{AH12}. We include here a simple proof using Clifford's Theorem. 
\begin{lem}\label{cliffres}
Let $\mathsf{S}$ be a supercharacter theory of $G$, and let $N$ be $\mathsf{S}$-normal. For each $\psi\in\mathrm{Ch}(\mathsf{S}_N)$, we have
\[\mathrm{Ind}_N^G(\psi)=\sigma_{\mathrm{Irr}(G\mid \psi)}.\]
\end{lem}
\begin{proof}
Let $\xi\in X=\mathrm{Irr}(G\mid\psi)$ and let $\vartheta$ be in irreducible constituent of $\mathrm{Res}^G_N(\xi)$. By Clifford's Theorem, 
\[\mathrm{Res}^G_N(\xi)=e\sum_{\eta\in\mathcal{O}}\eta,\,\ \text{where}\ \,e=\inner{\vartheta,\mathrm{Res}^G_N(\xi)}\ \,\ \text{and}\ \,\mathcal{O}=\mathrm{orb}_G(\vartheta).\]
By Frobenius reciprocity, we have
\begin{align*}\inner{\mathrm{Ind}_N^G(\psi),\xi}&=\inner{\psi,\mathrm{Res}^G_N(\xi)}\\
&=e\sum_{\eta\in \mathcal{O}}\inner{\psi,\eta}=e\sum_{\eta\in \mathcal{O}}\eta(1)=\xi(1),
\end{align*}
as desired.
\end{proof}

\begin{proof}[Proof of Theorem~\ref{degreediv}]
First, we prove {\bf(1)}.

If $N\le\ker(\chi)$, then $\psi=\mathds{1}_N$ and the result is obvious. So assume that $N\nsubseteq\ker(\chi)$. 

Write $X=\mathrm{Irr}(\chi)$, and let $\{\psi_i\mathrel{:}1\le i\le m\}$ be a set of distinct representatives of the $G$-orbits of the irreducible constituents of $\mathrm{Res}^G_N(\chi)$. For each $i$, write
\[C_i=\mathrm{Irr}(G\mid\psi_i)\cap X\]and note that $X$ is the disjoint union of the $C_i$. For each $1\le i\le m$ and for each $\vartheta\in\mathrm{Irr}(G\mid\psi_i)$, write
\[e_{i,\vartheta}=\inner*{\psi_i,\mathrm{Res}^G_N(\vartheta)}\ \,\text{and}\ \,t_i=\norm{\mathrm{orb}_G(\psi_i)}.\]Fix $i$ and note that $t_ie_{i,\vartheta}=\vartheta(1)/\psi_i(1)$ for each $\vartheta\in\mathrm{Irr}(G\mid\psi_i)$ by Clifford's Theorem. Then for each $\xi\in\mathrm{Irr}(G\mid\psi_i)$, we have
\begin{align*}
\inner*{\mathrm{Ind}^G_N\mathrm{Res}^G_N(\sigma_{C_i}),\xi}&=\sum_{\vartheta\in C_i}\vartheta(1)\inner*{\mathrm{Res}^G_N(\vartheta),\mathrm{Res}^G_N(\xi)}\\
&=\sum_{\vartheta\in C_i}\vartheta(1)t_ie_{i,\vartheta} e_{i,\xi} =\sum_{\vartheta\in C_i}\vartheta(1)(\vartheta(1)/\psi_i(1)) e_{i,\xi}=\frac{\sigma_{C_i}(1)e_{i,\xi}}{\psi_i(1)}.
\end{align*}
Hence, we have
\begin{align*}\inner*{\mathrm{Ind}^G_N\mathrm{Res}^G_N(\sigma_{C_i}),\xi/\xi(1)}&=\frac{\sigma_{C_i}(1)e_{i,\xi}}{\psi_i(1)\xi(1)}\\
&=\frac{\sigma_{C_i}(1)e_{i,\xi}\psi_i(1)}{\xi(1)\psi_i(1)^2}=\frac{\sigma_{C_i}(1)}{t_i\psi_i(1)^2},
\end{align*}
and so it follows that
\[\mathrm{Ind}^G_N\mathrm{Res}^G_N(\chi)=\sum_{i=1}^m\frac{\sigma_{C_i}(1)}{t_i\psi_i(1)^2}\sigma_{\mathrm{Irr}(G\mid\psi_i)}.\]
From \cite[Definition 6.3]{AH12}, we have that $\mathrm{Res}^G_N(\chi)=\alpha\psi$, where $\alpha$ is the rational number $\alpha=\chi(1)/\psi(1)$. Also, by Lemma~\ref{cliffres}, we have
\[\mathrm{Ind}_N^G(\psi)=\sigma_{\mathrm{Irr}(G\mid\psi)}=\sum_{i=1}^m\sigma_{\mathrm{Irr}(G\mid\psi_i)}.\]
Therefore, we have
\[\mathrm{Ind}^G_N\mathrm{Res}^G_N(\chi)=\alpha\mathrm{Ind}^G_N(\psi)=\alpha\sum_{i=1}^m\sigma_{\mathrm{Irr}(G\mid\psi_i)}.\]
From this observation, it follows that 
\[\alpha=\frac{\sigma_{C_1}(1)}{t_1\psi_1(1)^2}=\frac{\sigma_{C_2}(1)}{t_2\psi_2(1)^2}=\dotsb=\frac{\sigma_{C_m}(1)}{t_m\psi_m(1)^2}.\]
Finally, note that for $\xi\in C_1$, 
\[\alpha_\xi\coloneqq\frac{\xi(1)^2}{t_1\psi_1(1)^2}=\frac{(\xi(1)/\psi_1(1))^2}{t_1}=\frac{(t_1e_{1,\xi})^2}{t_1}=t_1e_{1,\xi}^2\in\mathbb{Z},\]
and so 
\[\frac{\chi(1)}{\psi(1)}=\alpha=\sum_{\xi\in C_1}\alpha_\xi\in\mathbb{Z},\]
as desired.

Now, we prove {\bf(2)}.

Let $g\in G\setminus N$, where $N\lhd_{\mathsf{S}}G$. Let $\{g_1,\dotsc,g_t\}$ be a set of distinct coset representatives of the elements of $\mathrm{cl}_{\mathsf{S}_{G/N}}(g)$. Define $N_i=\{n\in N\mid g_in\in \mathrm{cl}_{\mathsf{S}}(g)\}$ for each $i$, and note that $\mathrm{cl}_{\mathsf{S}}(g)$ is the disjoint union of the sets $g_iN_i$, $1\le i\le t$.
It follows from \cite[Lemma 1.2]{LM90} that $\norm{N_i}=\norm{N_j}$ for all $i,j$. Hence $\norm{\mathrm{cl}_{\mathsf{S}_{G/N}}(gN)}=t$, which is a divisor of $t\norm{N_1}=\norm{\mathrm{cl}_{\mathsf{S}}(g)}$. 
\end{proof}

\section{Central elements and commutators}\label{centralcomm}
In this section, we prove several properties of the subgroups $\mathbf{Z}(\mathsf{S})$ and $[G,\mathsf{S}]$. Many of these will seem very familiar as they are direct generalizations of well-known results.

Let $\mathsf{S}\in\mathrm{SCT}(G)$. We define the subgroup $\mathbf{Z}(\mathsf{S})$, by

\[\mathbf{Z}(\mathsf{S})=\{g\in G:\norm{\mathrm{cl}_{\mathsf{S}}(g)}=1\}.\]
The fact that $\mathbf{Z}(\mathsf{S})$ is a subgroup, and hence an $\mathsf{S}$-normal subgroup, follows from the fact that products of $\mathsf{S}$-classes are unions of $\mathsf{S}$-classes. This fact was observed by Hendrickson in \cite{AH08}, as was the following result.

\begin{lem}[{\normalfont cf. \cite[Lemma 2.3]{AH08}}]\label{classaction}
Let $z\in\mathbf{Z}(\mathsf{S})$ and let $g\in G$. Then $z\mathrm{cl}_{\mathsf{S}}(g)=\mathrm{cl}_{\mathsf{S}}(zg)$.
\end{lem}
This lemma allows us to prove the following result, which strengthens Theorem~\ref{degreediv} in the case that $N\le \mathbf{Z}(\mathsf{S})$. 

\begin{prop}\label{otherclassdivide}
Let $N$ be an $\mathsf{S}$-normal subgroup of $G$ satisfying $N\le \mathbf{Z}(\mathsf{S})$. Then $\norm{\mathrm{cl}_{\mathsf{S}}(g)}/\norm{\mathrm{cl}_{\mathsf{S}_{G/N}}(gN)}$ divides $\norm{N}$ for each $g\in G$.
\end{prop}
\begin{proof}
Let $K\in\mathrm{Cl}(\mathsf{S})$. Note that $N$ acts on $\mathrm{Cl}(\mathsf{S})$ by Lemma~\ref{classaction}. Also note that for every $n,n'\in N$, we have $nK=n'K$ if and only if we have $n^{-1}n'\in\mathrm{Stab}_N(K)$. It follows that
\[\norm{NK}=\norm{\{nK: n\in N\}}\norm{K}=\norm{\mathrm{orb}_N(K)}\norm{K}.\]
But also, for every $k,k'\in K$, we have $Nk=Nk'$ if and only if we have $\pi(k)=\pi(k')$, where $\pi:G\to G/N$ is the canonical projection, so
\[\norm{NK}=\norm{N}\norm{\{Nk: k\in K\}}=\norm{N}\norm{\pi(K)}.\]
Hence
\[\norm{K}=\frac{\norm{N}\norm{\pi(K)}}{\norm{\mathrm{orb}_N(K)}}=\norm{\pi(K)}\norm{\mathrm{Stab}_N(K)},\]
and the result follows.
\end{proof}

Next we will establish some character theoretic properties of $\mathbf{Z}(\mathsf{S})$, but first we need the following generalization of column orthogonality. 

\begin{thm}\label{column}Let $\mathsf{S}$ be a supercharacter theory of $G$. For each $g,h\in G$, we have
\[ \sum_{\chi\in\mathrm{Ch}(\mathsf{S})}\frac{\chi(g)\overline{\chi(h)}}{\chi(1)}=\frac{\norm{G}}{\norm{\mathrm{cl}_{\mathsf{S}}(g)}}\]
if $h\in\mathrm{cl}_{\mathsf{S}}(g)$, and is 0 otherwise.
\end{thm}
\begin{proof}Let $\mathrm{Ch}(\mathsf{S})=\{\chi_1,\chi_2,\dotsc,\chi_\ell\}$, let $\{g_1,g_2,\dotsc,g_\ell\}$ be representatives for the distinct $\mathsf{S}$-classes of $G$, and write $K_i=\mathrm{cl}_{\mathsf{S}}(g_i)$ for each $i$. Let $A=(\chi_i(g_{j}))$, let $D=(\delta_{i,j}\norm{K_i})$ and let $A^\dagger$ denote the conjugate transpose of $A$. Then since
\[\inner{\chi_i,\chi_j}=\sum_{\substack{\xi\in \mathrm{Irr}(\chi_i)\\\vartheta\in \mathrm{Irr}(\chi_j)}}\xi(1)\vartheta(1)\inner{\xi,\vartheta}=\delta_{i,j}\chi_i(1),\]
we have
\[ADA^\dagger=\norm{G}T^2,\ \,\text{where}\ \,T=\left(\delta_{i,j}\sqrt{\chi_i(1)}\right).\]
It follows that \[1=Y(ADA^\dagger)Y=(YA)D(YA)^\dagger=D(YA)^\dagger(YA),\ \text{where}\ \,Y=\frac{1}{\sqrt{\norm{G}}}T^{-1}.\]
Rewriting, we obtain
\[1=D(YA)^\dagger(YA)=\left(\frac{\norm{K_j}}{\norm{G}} \sum_{k=1}^\ell\frac{\overline{\chi_k(g_{i})}\chi_k(g_{j})}{\chi_k(1)}\right)\]
which gives
\[ \sum_{k=1}^\ell\frac{\overline{\chi_k(g_{i})}\chi_k(g_{j})}{\chi_k(1)}=\delta_{i,j}\frac{\norm{G}}{\norm{K_j}},\]
as required.
\end{proof}

\begin{lem}\label{centerrestriction}Let $\mathsf{S}$ be a supercharacter theory of $G$, and let $g\in\mathbf{Z}(\mathsf{S})$. For each $\chi\in\mathrm{Ch}(\mathsf{S})$, there exists $\vartheta_\chi\in\mathrm{Irr}(\mathbf{Z}(\mathsf{S}))$ such that $\mathrm{Res}^G_{\mathbf{Z}(\mathsf{S})}(\chi)=\chi(1)\vartheta_\chi$.\end{lem}
\begin{proof}
Write $\mathcal{A}=\mathbb{C}$-$\mathrm{span}\{ \widehat{K}:K\in\mathrm{Cl}(\mathsf{S})\}$. By \cite[Theorem 2.2]{ID07},  $\mathcal{A}$ is a subalgebra of $\mathbf{Z}(\mathbb{C}{G})$, and $\{f_\xi:\xi\in\mathrm{Ch}(\mathsf{S})\}$ forms a basis for $\mathcal{A}$, where
\[f_\xi=\frac{1}{\norm{G}}\sum_{g\in G}\overline{\xi(g)}g\]
for each $\xi\in\mathrm{Ch}(\mathsf{S})$; i.e., $f_\xi$ is the sum of the primitive central idempotents $\frac{\vartheta(1)}{\norm{G}}\sum_{g\in G}\overline{\vartheta(g)}g$ for $\vartheta\in\mathrm{Irr}(\xi)$. For each $z\in\mathrm{A}$, write
\[z=\sum_{\xi\in\mathrm{Ch}(\mathsf{S})}\omega_\xi(z)f_\xi.\]
Then if $\mathfrak{X}$ is any representation affording $\chi$, we have $\chi(z)=\mathrm{tr}(\mathfrak{X}(z))=\omega_\chi(z)\chi(1)$. Since $\omega_\chi$ is an algebra homomorphism (see the proof of \cite[Theorem 2.4]{ID07}) and $\mathbf{Z}(\mathsf{S})\subseteq\mathcal{A}$, it follows that $\omega_\chi\vert_{\mathbf{Z}_{\mathsf{S}}(G)}$ is a group homomorphism $G\to\mathbb{C}^\times$. So $\mathrm{Res}^G_N(\chi)=\chi(1)\vartheta_\chi$, where $\vartheta_\chi=\omega_\chi\vert_{\mathbf{Z}_{\mathsf{S}}(G)}$ is a linear character of $\mathbf{Z}(\mathsf{S})$.
\end{proof}

Using Theorem~\ref{column} and Proposition~\ref{centerrestriction}, we are ready to prove the following result, which is a direct generalization of \cite[Lemma 2.27, Corollary 2.28]{MI76}.
\begin{prop}\label{centercharacter}
Let $\mathsf{S}$ be a supercharacter theory of $G$, let $\chi\in\mathrm{Ch}(\mathsf{S})$ and write $Z=\mathbf{Z}(\chi)$. Then the following hold:
\begin{enumerate}[label={\bf(\arabic*)}]\openup 5pt
\item$Z\lhd_{\mathsf{S}}G$;
\item$\mathrm{Res}^G_Z(\chi)=\chi(1)\vartheta_\chi$ for some $\vartheta_\chi\in\mathrm{Irr}(Z)$;
\item$\mathbf{Z}(\mathsf{S})=\bigcap_{\chi\in\mathrm{Ch}(S)}\mathbf{Z}(\chi)$;
\item$Z/\ker(\chi)=\mathbf{Z}(\mathsf{S}_{G/\ker(\chi)})$, and is cyclic.
\end{enumerate}
\end{prop}
\begin{proof}
First, note that {\bf(2)} and the fact that $Z\le G$ follow from \cite[Lemma 2.27 (b,c)]{MI76}, which asserts the same facts for any character $\chi$ of $G$. Since $\chi$ is constant on $\mathsf{S}$-classes, $Z$ is a union of $\mathsf{S}$-classes, which proves {\bf(1)}. 

Next we prove {\bf(3)}.  Let $\chi\in\mathrm{Ch}(\mathsf{S})$. For $g\in\mathbf{Z}(\mathsf{S})$, we have $\chi(g)=\chi(1)\vartheta_\chi(g)$ for some $\vartheta_\chi\in\mathrm{Irr}(\mathbf{Z}(\mathsf{S}))$ by Proposition~\ref{centerrestriction}. Since $\norm{\vartheta_\chi(g)}=1$, it follows that $\mathbf{Z}(\mathsf{S})\le\mathbf{Z}(\chi)$.

Suppose that $g\in\mathbf{Z}(\chi)$ for each $\chi\in\mathrm{Ch}(\mathsf{S})$. By Theorem~\ref{column}, we have
\begin{align*}\frac{\norm{G}}{\norm{\mathrm{cl}_{\mathsf{S}}(g)}}&= \sum_{\chi\in\mathrm{Ch}(\mathsf{S})}\frac{\norm{\chi(g)}^2}{\chi(1)}= \sum_{\chi\in\mathrm{Ch}(\mathsf{S})}\frac{\chi(1)^2}{\chi(1)}=\sum_{\chi\in\mathrm{Ch}(\mathsf{S})}\chi(1)=\norm{G}.
\end{align*}
Hence $\norm{\mathrm{cl}_{\mathsf{S}}(g)}=1$, and so $g\in\mathbf{Z}(\mathsf{S})$. This finishes the proof of {\bf(3)}.

Finally we prove {\bf(4)}.  Let $\chi\in\mathrm{Ch}(\mathsf{S})$.  The fact that $Z/\ker(\chi)$ is cyclic follows from \cite[Lemma 2.27 (d)]{MI76}, so we need only prove that $Z/\ker(\chi)$ coincides with $\mathbf{Z}(\mathsf{S}_{G/\ker(\chi)})$. Since $\ker(\chi)\le\ker(\psi)\le\mathbf{Z}(\psi)$ for each $\psi\in\mathrm{Ch}(\mathsf{S}/\ker(\chi))$, we have
\begin{align*}\mathbf{Z}(\mathsf{S}_{G/\ker(\chi)})&=\bigcap_{\psi\in\mathrm{Ch}(\mathsf{S}/\ker(\chi))}\mathbf{Z}(\mathrm{Def}^G_{G/\ker(\chi)}(\psi))\\
&=\bigcap_{\psi\in\mathrm{Ch}(\mathsf{S}/\ker(\chi))}\mathbf{Z}(\psi)/\ker(\chi)\le\mathbf{Z}(\chi)/\ker(\chi).\end{align*}

To show the other containment, let $g\in\mathbf{Z}(\chi)$, and $k\in\mathrm{cl}_{\mathsf{S}}(g)$. Let $\vartheta_\chi\in\mathrm{Irr}(\mathbf{Z}(\chi))$ be such that $\mathrm{Res}^G_Z(\chi)=\chi(1)\vartheta_\chi$.  Since $\chi(g)=\chi(k)$, it follows that $\vartheta_\chi(g)=\vartheta_\chi(k)$; hence $g^{-1}k\in\ker(\vartheta_\chi)=\ker(\chi)$. Thus $g\ker(\chi)=k\ker(\chi)$, and it follows that $g\ker(\chi)\in\mathbf{Z}(\mathsf{S}_{G/\ker(\chi)})$, which completes the proof of {\bf(4)}.
\end{proof}

Let $H\le G$. Recall that the subgroup $[H,\mathsf{S}]$ is defined by
\[[H,\mathsf{S}]=\inner{g^{-1}k:g\in H,\ k\in\mathrm{cl}_{\mathsf{S}}(g)}.\]
By definition, $[G,\mathsf{S}]$ is a normal subgroup of $G$ containing $[G,G]$. Less obviously, $[G,\mathsf{S}]$ is also $\mathsf{S}$-normal. To prove this, we will need the following lemma.
\begin{lem}\label{lattice}
Let $\mathsf{S}$ be a supercharacter theory of $G$, and let $H,N$ be $\mathsf{S}$-normal. Then $H\cap N$ and $HN$ are both $\mathsf{S}$-normal.
\end{lem}
\begin{proof}
Since $H$ and $N$ are intersections of kernels of $\mathsf{S}$-characters, so is $H\cap N$. 

For each $h\in H$ and $n\in N$, the set $\mathrm{cl}_{\mathsf{S}}(h)\mathrm{cl}_{\mathsf{S}}(n)$ is a union of $\mathsf{S}$-classes. Hence so is $HN$, as $H$ and $N$ are both unions of $\mathsf{S}$-classes.
\end{proof}
\begin{prop}\label{commnormal}
Let $\mathsf{S}$ be a supercharacter theory of $G$. Then $[G,\mathsf{S}]$ is $\mathsf{S}$-normal.
\end{prop}
\begin{proof}
Let $I=\{g^{-1}k:g\in G,\ k\in\mathrm{cl}_{\mathsf{S}}(g)\}$.  Define $C_K=\bigcup_{g\in K}\{g^{-1}k:k\in K\}$ for each $K\in\mathrm{Cl}(\mathsf{S})$. Notice that $C_K$ is precisely the set of elements of $G$ that appear in $\widehat{K^{-1}}\widehat{K}$, which is a union of $\mathsf{S}$-classes, say $C_{K,i}$, $1\le i\le \ell_K$. By \cite[Proposition 23.6]{HW64}, we have that $\inner{C_{K,i}}\lhd_{\mathsf{S}}G$ for each $K$ and $i$. The result follows since
\[[G,\mathsf{S}]=\inner{I}=\prod_{K\in\mathrm{Cl}(\mathsf{S})}\inner{C_K}=\prod_{K\in\mathrm{Cl}(\mathsf{S})}\prod_{i=1}^{\ell_K}\inner{C_{K,i}},\]
which is $\mathsf{S}$-normal by Lemma~\ref{lattice}.
\end{proof}
\begin{cor}\label{snormalcommutator}
Let $N\lhd_{\mathsf{S}}G$. Then $[N,\mathsf{S}]=\inner{g^{-1}k:g\in N,\ k\in\mathrm{cl}_{\mathsf{S}}(g)}$ is $\mathsf{S}$-normal.
\end{cor}

Moreover, we can prove that $N\lhd_{\mathsf{S}}G$ if and only if $[N,\mathsf{S}]\le N$. To do this, we define for any $H\le G$ the subgroup $H^\mathsf{S}$ to be the intersection of all $\mathsf{S}$-normal subgroups of $G$ containing $H$.

\begin{prop}\label{sclosure}
Let $H\le G$ and $\mathsf{S}$ be a supercharacter theory of $G$. Then $H^\mathsf{S}=H[H,\mathsf{S}]$.
\end{prop}
\begin{proof}
First, we show that $[H,\mathsf{S}]\lhd G$. It suffices to show that $g^{-1}h^{-1}kg$ lies in $[H,\mathsf{S}]$ for all $h\in H$, $k\in \mathrm{cl}_{\mathsf{S}}(h)$ and $g\in G$. Let $h\in H$, $k\in \mathrm{cl}_{\mathsf{S}}(h)$ and $g\in G$. Then, we have
\[g^{-1}h^{-1}kg=g^{-1}h^{-1}ghh^{-1}g^{-1}kg=\left(h^{-1}g^{-1}hg\right)^{-1}h^{-1}g^{-1}kg.\]
Since $h,k\in \mathrm{cl}_{\mathsf{S}}(h)$ and $\mathrm{cl}_{\mathsf{S}}(h)$ is a union of $G$-classes, $g^{-1}hg$ and $g^{-1}kg$ also lie in $\mathrm{cl}_{\mathsf{S}}(h)$. It follows that $g^{-1}h^{-1}kg\in[H,\mathsf{S}]$ and so $[H,\mathsf{S}]\lhd G$. Thus $H[H,\mathsf{S}]\le G$ and clearly $H[H,\mathsf{S}]\le H^{\mathsf{S}}$. 

For the reverse containment, note that we have
\[H^{\mathsf{S}}=\inner{\mathrm{cl}_{\mathsf{S}}(h): h\in H}\le H[H,\mathsf{S}].\]
The result follows.
\end{proof}
\begin{cor}
Let $\mathsf{S}$ be a supercharacter theory of $G$. A subgroup $H\le G$ is $\mathsf{S}$-normal if and only if $[H,\mathsf{S}]\le H$.
\end{cor}
The next result describes exactly which $\mathsf{S}$-characters are linear. 
\begin{prop}\label{gsproperty}
Let $\mathsf{S}$ be a supercharacter theory of $G$, and let $N$ be $\mathsf{S}$-normal. Then $[G,\mathsf{S}]\le N$ if and only if $\mathbf{Z}(\mathsf{S}_{G/N})=G/N$.
\end{prop}
\begin{proof}
Let $\pi:G\to G/N$ be the canonical projection. Then $\mathbf{Z}(\mathsf{S}_{G/N})=G/N$ if and only if $\pi(g)=\pi(k)$ for every $g\in G$ and $k\in\mathrm{cl}_{\mathsf{S}}(g)$. This happens if and only if $g^{-1}k\in N$ for all $g\in G$ and $k\in\mathrm{cl}_{\mathsf{S}}(g)$, which is equivalent to the condition $[G,\mathsf{S}]\le N$. 
\end{proof}
As a corollary to this result, one has that an $\mathsf{S}$-character $\chi$ is linear if and only if $[G,\mathsf{S}]\le\ker(\chi)$. So the supercharacters of the deflated theory $\mathsf{S}_{G/[G,\mathsf{S}]}$ can naturally be considered as the set of linear $\mathsf{S}$-characters. The set of these characters, denoted $\mathrm{Ch}(\mathsf{S}/[G,\mathsf{S}])$, coincides with $\mathrm{Irr}(G/[G,\mathsf{S}])$ and hence forms a group. Moreover, this group acts on $\mathrm{Ch}(\mathsf{S})$. More generally, so does $\mathrm{Ch}(\mathsf{S}/N)$ does for any $N\lhd_{\mathsf{S}}G$ satisfying $[G,\mathsf{S}]\le N$.
\begin{lem}\label{linearaction}
Let $\mathsf{S}$ be a supercharacter theory of $G$, and let $N$ be an $\mathsf{S}$-normal subgroup of $G$ satisfying $[G,\mathsf{S}]\le N$. Then $\chi\lambda$ is an $\mathsf{S}$-character for each $\lambda\in\mathrm{Ch}(\mathsf{S}/N)$ and $\chi\in\mathrm{Ch}(\mathsf{S})$. In particular, the group $\mathrm{Ch}(\mathsf{S}/N)=\mathrm{Irr}(G/N)$ acts on $\mathrm{Ch}(\mathsf{S})$.
\end{lem}
\begin{proof}
Let $\chi\in\mathrm{Ch}(\mathsf{S})$, and let $\lambda\in\mathrm{Ch}(\mathsf{S}/N)$. Since $\lambda$ is linear, $\chi\lambda$ is a character of $G$, which is constant on the $\mathsf{S}$-classes of $G$. So $\chi\lambda$ is a sum of $\mathsf{S}$-characters. Let $\xi\in\mathrm{Ch}(\mathsf{S})$ such that $\mathrm{Irr}(\xi)\subseteq\mathrm{Irr}(\chi\lambda)$. Then $\mathrm{Irr}(\xi\lambda^{-1})\subseteq\mathrm{Irr}(\chi)$, and since $\xi\lambda^{-1}$ is also a sum of $\mathsf{S}$-characters, it follows that $\xi=\chi\lambda$.
\end{proof}

\begin{prop}\label{actiondivide}
Let $N$ be an $\mathsf{S}$-normal subgroup satisfying $[G,\mathsf{S}]\le N$. Let $\chi\in\mathrm{Ch}(\mathsf{S}\mid N)$ and suppose that $\psi\in\mathrm{Ch}(\mathsf{S}_N)$ satisfies $\inner{\mathrm{Res}^G_N(\chi),\psi}>0$. Then $\chi(1)/\psi(1)$ divides $\norm{G:N}$. 
\end{prop}

\begin{proof}
Since $[G,\mathsf{S}]\le N$, $\Lambda=\mathrm{Ch}(\mathsf{S}/N)$ acts on $\mathrm{Ch}(\mathsf{S})$ by Lemma~\ref{linearaction}. Consider the set $C=\{\psi\lambda:\ \psi\in X,\ \lambda\in\Lambda\}$, where $X=\mathrm{Irr}(G;\chi)$. On the one hand, $C$ is exactly the set of constituents of $\mathrm{Ind}_N^G(\psi)$. By \cite[Lemma 3.4]{AH08}, we conclude that
\[\sigma_C(1)=\mathrm{Ind}_N^G(\psi)(1)=\norm{G:N}\psi(1).\]
On the other hand, we have $C=\bigcup_{\lambda\in \Lambda}X^\lambda$. Since $\mathrm{Irr}(G;\chi^\lambda)\cap \mathrm{Irr}(G;\chi)=\varnothing$ whenever $\chi^\lambda\neq \chi$ and $\chi^\lambda(1)=\chi(1)$ for each $\lambda\in\Lambda$, we have
\[\sigma_C(1)=\norm{\mathrm{orb}_{\Lambda}(\chi)}\chi(1).\]
Thus, we have
\[\chi(1)=\frac{\norm{G:N}\psi(1)}{\norm{\mathrm{orb}_{\Lambda}(\chi)}}=\norm{\mathrm{Stab}_{\Lambda}(\chi)}\psi(1).\]
The result follows as $\norm{\mathrm{Stab}_{\Lambda}(\chi)}$ divides $\norm{G:N}$.
\end{proof}
 
We conclude this section with the following result, which shows how the subgroups $\mathbf{Z}(\mathsf{S})$ and $[G,\mathsf{S}]$ behave with the lattice structure of $\mathrm{SCT}(G)$.

\begin{lem}\label{basicproperties}Let $\mathsf{S}$ and $\mathsf{T}$ be supercharacter theories of $G$. The following hold:

\begin{enumerate}[label={\bf(\arabic*)}]\openup5pt
\item if $\mathsf{S}\preccurlyeq\mathsf{T}$, then $\mathbf{Z}(\mathsf{T})\le\mathbf{Z}(\mathsf{S})$ and $[G,\mathsf{S}]\le[G,\mathsf{T}]$;

\item$\mathbf{Z}(\mathsf{S})\mspace{2mu}\mathbf{Z}(\mathsf{T})\le\mathbf{Z}(\mathsf{S}\wedge\mathsf{T})$ and $[G,\mathsf{S}\wedge\mathsf{T}]\le[G,\mathsf{S}]\cap[G,\mathsf{T}]$;

\item$\mathbf{Z}(\mathsf{S}\vee\mathsf{T})=\mathbf{Z}(\mathsf{S})\cap\mathbf{Z}(\mathsf{T})$ and $[G,\mathsf{S}\vee\mathsf{T}]=[G,\mathsf{S}][G,\mathsf{T}]$.
\end{enumerate}
\end{lem}
\begin{proof}
If $\mathsf{S}\preccurlyeq\mathsf{T}$, then every class of $\mathsf{S}$ is contained in a class of $\mathsf{T}$. Thus any element of $\mathbf{Z}(\mathsf{T})$ must also be an element of $\mathbf{Z}(\mathsf{S})$, just by size considerations. Also, note that the set $\{g^{-1}k:g\in G,\ k\in\mathrm{cl}_{\mathsf{S}}(g)\}$ is a subset of $\{g^{-1}k:g\in G,\ k\in\mathrm{cl}_{\mathsf{T}}(g)\}$. 
It follows that {\bf(1)} holds.

Statement {\bf(2)} follows immediately from {\bf(1)}. 

By {\bf(1)}, we have $\mathbf{Z}(\mathsf{S}\vee\mathsf{T})\le\mathbf{Z}(\mathsf{S})\cap\mathbf{Z}(\mathsf{T})$. To obtain the other containment, let $g\in \mathbf{Z}(\mathsf{S})\cap\mathbf{Z}(\mathsf{T})$. Then $\{g\}\in\mathrm{Cl}(\mathsf{S})\cap\mathrm{Cl}(\mathsf{T})$, which is a subset of $\mathrm{Cl}(\mathsf{S}\vee\mathsf{T})$; the first statement of {\bf(3)} follows.

Finally, we prove that $[G,\mathsf{S}\vee\mathsf{T}]=[G,\mathsf{S}][G,\mathsf{T}]$. By {\bf(1)}, $[G,\mathsf{S}\vee\mathsf{T}]\le [G,\mathsf{S}][G,\mathsf{T}]$. To see the other inclusion, let $g\in G$, and let $k\in C=\mathrm{cl}_{\mathsf{S}\vee\mathsf{T}}(g)$. Recall that $\mathrm{Cl}(\mathsf{S}\vee\mathsf{T})=\mathrm{Cl}(\mathsf{S})\vee\mathrm{Cl}(\mathsf{T})$, so there exists $\{K_i\}_{i=1}^m\subseteq\mathrm{Cl}(\mathsf{S})$ and $\{L_i\}_{i=1}^\ell\subseteq\mathrm{Cl}(\mathsf{T})$ for which
\[C=\bigcup_{i=1}^mK_i=\bigcup_{i=1}^\ell L_i.\]
Suppose that $g\in K_{a}\cap L_c$ and $k\in K_{b}\cap L_d$. Then we may choose multisets $\{a=i_1,i_2,\dotsc,i_f=b\}\subseteq\{1,2,\dotsc,m\}$ and $\{c=j_1,j_2,\dotsc j_{f-1}=d\}\subseteq\{1,2,\dotsc,\ell\}$ so that 
\[K_{i_s}\cap L_{j_s}\neq\varnothing,\ \,\text{and}\ \,K_{i_{s+1}}\cap L_{j_s}\neq\varnothing,\]
for each $1\le s\le f-1$. Choose 
\[k_1\in K_{i_2}\cap L_{j_1},\ k_2\in K_{i_2}\cap L_{j_2},\ k_3\in K_{i_3}\cap L_{j_2},\ \dotsc,\ k_{2f-4}\in K_{i_{f-1}}\cap L_{i_{f-1}}.\] Then if $k_0=g$, which is in $K_{i_1}\cap J_{j_1}$, and $k_{2f-3}=k$, which is in $K_{i_f}\cap L_{i_{f-1}}$, then $k_i^{-1}k_{i+1}\in[G,\mathsf{T}]$ for each even $i$, and $k_i^{-1}k_{i+1}\in[G,\mathsf{S}]$ for each odd $i$. It follows that
\[g^{-1}k=(g^{-1}k_1)(k_1^{-1}k_2)\dotsb(k_{2f-3}^{-1}k_{2f-4})(k_{2f-4}^{-1}k)\in[G,\mathsf{S}][G,\mathsf{T}].\]
Since $[G,\mathsf{S}\vee\mathsf{T}]$ is generated by elements of the form $g^{-1}k$ for $k\in\mathrm{cl}_{\mathsf{S}\vee\mathsf{T}}(g)$, the result follows. This completes the proof of {\bf(3)}.
\end{proof}
\section{$\mathsf{S}$-chief series and $\mathsf{S}$-central series}
In this section, we introduce certain $\mathsf{S}$-normal series, i.e., normal series consisting of $\mathsf{S}$-normal subgroups, where $\mathsf{S}$ is a supercharacter theory of $G$. In particular, we discuss $\mathsf{S}$-central series, and give several characterizations of when $G$ has an $\mathsf{S}$-central series. First though, we discuss an analog of chief series.

We will say that a supercharacter theory $\mathsf{S}$ of $G$ is {\bf simple} if $G$ has no nontrivial proper $\mathsf{S}$-normal subgroups. Let $N_1$ be a maximal $\mathsf{S}$-normal subgroup of $G$. Then the induced supercharacter theory $\mathsf{S}_{G/N_1}$ is simple by Lemma~\ref{latticeisothm}. Continuing this way, we can construct an $\mathsf{S}$-normal series
\[G\ge N_1\ge\dotsb\ge N_s=1,\]
where $\mathsf{S}_{N_i/N_{i+1}}$ is simple for each $i$. We call such a series an $\mathsf{S}$-{\bf chief} series.

Note that the set $\mathrm{Norm}(\mathsf{S})$ of $\mathsf{S}$-normal subgroups of $G$ forms a (modular) sublattice of the lattice normal subgroups of $G$ by Lemma~\ref{lattice}. It is well-known that the Jordan--H\"{o}lder Theorem applies in this situation.

\begin{thm}[Jordan--H\"{o}lder]\label{JH}
Let $G\ge N_1\ge\dotsb\ge N_s=1$ and $G\ge H_1\ge\dotsb\ge H_\ell=1$ be two $\mathsf{S}$-chief series for $G$. Then $s=\ell$, and there is $\sigma\in\mathrm{Sym}(\{1,2,\dotsc,s-1\})$ such that $N_i/N_{i+1}\simeq H_{\sigma(i)}/H_{\sigma(i)+1}$ for each $1\le i\le s-1$.
\end{thm}

If $N$ and $H$ are $\mathsf{S}$-normal in $G$, and $H/N$ is simple, we will call $H/N$ an $\mathsf{S}$-chief factor of $G$.  As a corollary of Theorem~\ref{JH}, we have that the set of isomorphism classes of $\mathsf{S}$-chief factors is independent of the choice of $\mathsf{S}$-chief series.

A finite nilpotent group has many equivalent characterizations, and several of these can be stated in terms of central series and chief series. Before giving various analogous characterizations of $\mathsf{S}$-nilpotence via $\mathsf{S}$-normal series, we make the following observation. By Proposition~\ref{gsproperty}, we have 
\[\gamma_{i}(\mathsf{S})/\gamma_{i+1}(\mathsf{S})=\mathbf{Z}(\mathsf{S}_{\gamma_{i}(\mathsf{S})/\gamma_{i+1}(\mathsf{S})})\le\mathbf{Z}(\mathsf{S}_{G/\gamma_{i+1}(\mathsf{S})})\le\mathbf{Z}(G/\gamma_{i+1}(\mathsf{S}))\]
for each $i\ge 1$; it follows that the lower $\mathsf{S}$-central series is indeed a central series if it is an $\mathsf{S}$-central series.

\begin{prop}\label{nilclassification}
Let $\mathsf{S}$ be a supercharacter theory of $G$. The following are equivalent:
\begin{enumerate}[label = {\bf(\arabic*)}]\openup 5pt
\item the group $G$ has an $\mathsf{S}$-central series;
\item the lower $\mathsf{S}$-central series terminates at $1$, and has minimal length;
\item the upper $\mathsf{S}$-central series terminates at $G$, and has minimal length;
\item every $\mathsf{S}$-chief factor $N/H$ has prime order, and $\mathsf{S}_{N/H}=\mathsf{m}(N/H)$;
\item every $\mathsf{S}$-chief series is also an $\mathsf{S}$-central series. 
\end{enumerate}
\end{prop}
\begin{proof}
The proof that {\bf(1)}, {\bf(2)} and {\bf(3)} are equivalent is essentially identical to the proof of the analogous group theory result. We include this here only for completeness. Assume that $G$ has an $\mathsf{S}$-central series, and let $G=N_1\ge N_2\ge\dotsb\ge N_{c+1}=1$ be an $\mathsf{S}$-central series of minimal length, $c$. We prove that 
\begin{equation}\label{central1}N_{c+1-i}(\mathsf{S})\le\zeta_i(\mathsf{S})\end{equation}
for each $i\ge0$ by induction on $i$. Since we have $N_{c+1}(\mathsf{S})=1=\zeta_0(\mathsf{S})$, \eqref{central1} holds for $i=0$. Suppose then \eqref{central1} holds for some $i\ge0$. Since we have
\[N_{c-i}/N_{c-i+1}\le\mathbf{Z}(G/N_{r-i+1}),\]
and $N_{c+1-i}\le\zeta_i(\mathsf{S})$, it follows that
\[N_{c-i}\zeta_i(\mathsf{S})/\zeta_i(\mathsf{S})\le\mathbf{Z}(\mathsf{S}_{G/\zeta_i(\mathsf{S})})=\zeta_{i+1}(\mathsf{S})/\zeta_i(\mathsf{S}).\]
Hence $N_{c-i}$ is a subgroup of $\zeta_{i+1}(\mathsf{S})$, which proves \eqref{central1}. 

Since $N_1=G$, it follows from \eqref{central1} that $\zeta_{c}(\mathsf{S})=G$. Therefore the upper $\mathsf{S}$-central series is an $\mathsf{S}$-central series. This series must have length $c$ by minimality.

 A similar argument shows that
\begin{equation}\label{central2}\gamma_i(\mathsf{S})\le N_i\end{equation}
for each $i$; the inductive step follows from the fact that
\[\gamma_{i+1}(\mathsf{S})=[\gamma_i(\mathsf{S}),\mathsf{S})]\le[N_i,\mathsf{S}]\le N_{i+1}\]
whenever $\gamma_i(\mathsf{S})\le N_i$, since $N_i/N_{i+1}\le\mathbf{Z}(\mathsf{S}_{G/N_{i+1}})$. We deduce from \eqref{central2} that  the lower $\mathsf{S}$-central series is an $\mathsf{S}$-central series of minimal length $c$. It follows that {\bf(1)}, {\bf(2)}, and {\bf(3)} are equivalent.

Next we assume {\bf(3)} and prove {\bf(4)}. Let $1\le i\le c$, and let $H\lhd_{\mathsf{S}}G$ with $\zeta_{i-1}(\mathsf{S})\le H<\zeta_i(\mathsf{S})$. Let $N$ be an $\mathsf{S}$-normal cover of $H$ with $N\le \zeta_i(\mathsf{S})$, so that $N/H$ is a $\mathsf{S}$-chief factor. Since $N/H$ is isomorphic to $(N/\zeta_{i-1}(\mathsf{S}))/(H/\zeta_{i-1}(\mathsf{S}))$, it follows that
\[\mathsf{S}_{N/H}=\left(\mathsf{S}_{\zeta_i(\mathsf{S})/\zeta_{i-1}(\mathsf{S})}\right)_{N/H}=\mathsf{m}(\zeta_i(\mathsf{S})/\zeta_{i-1}(\mathsf{S}))_{N/H}=\mathsf{m}(N/H),\]
so $N/H\le\mathbf{Z}(\mathsf{S}_{G/H})$. Moreover, this means that $N/H$ is abelian and $M\lhd_{\mathsf{S}}G$ for each $H\le M\le N$, so $N/H$ must have prime order. Thus {\bf(4)} is proved.

Statement {\bf(5)} is immediate from {\bf(4)}. Also {\bf(1)} follows immediately from {\bf(5)} since $G$ has at least one $\mathsf{S}$-chief series.
\end{proof}

\section{$\mathsf{S}$-nilpotence}\label{nilpsection}
In this section, we outline the basic properties of $\mathsf{S}$-nilpotence. Recall that a group $G$ is called $\mathsf{S}$-nilpotent for a supercharacter theory $\mathsf{S}$ of $G$ if $G$ has an $\mathsf{S}$-central series. The equivalent conditions of Proposition~\ref{nilclassification} are therefore all equivalent to $G$ being $\mathsf{S}$-nilpotent.

Note that if $G$ is $\mathsf{S}$-nilpotent, then $G$ has a central series and must therefore be nilpotent. Actually a little more can be said. If $\mathsf{T}\preccurlyeq\mathsf{S}$ and $G$ is $\mathsf{S}$-nilpotent, then $[G,\mathsf{T}]\le[G,\mathsf{S}]$ and so it follows from Lemma~\ref{basicproperties} that $G$ is also $\mathsf{T}$-nilpotent. 

We now explore how the concept of $\mathsf{S}$-nilpotence behaves with respect to $\mathsf{S}$-normal subgroups.
\begin{prop}\label{firstfitting}
Let $H$ and $N$ be $\mathsf{S}$-normal. If $H$ is $\mathsf{S}_H$-nilpotent, and $N$ is $\mathsf{S}_N$-nilpotent, then $HN$ is $\mathsf{S}_{HN}$-nilpotent. 
\end{prop}
The proof of this result follows readily from the following lemma.
\begin{lem}\label{comsplit}
Let $\mathsf{S}\in\mathrm{SCT}(G)$. If $H$ and $N$ are $\mathsf{S}$-normal, then $[HN,\mathsf{S}]=[H,\mathsf{S}][N,\mathsf{S}]$. 
\end{lem}
\begin{proof}Since $H\lhd_{\mathsf{S}}G$, $\mathrm{cl}_{\mathsf{S}}(h)\subseteq H$ for all $h\in H$, which implies $[H,\mathsf{S}]\le[HN,\mathsf{S}]$. Similarly we have $[N,\mathsf{S}]\le[HN,\mathsf{S}]$;  it follows that $[H,\mathsf{S}][N,\mathsf{S}]\subseteq[HN,\mathsf{S}]$. 
We now show the reverse containment. Let $h\in H$ and $n\in N$. Then for each  $L\in\mathrm{Cl}(\mathsf{S})$, there exists a nonnegative integer $a_{h,n,L}$ such that
\[\widehat{\mathrm{cl}_{\mathsf{S}}(h)}\widehat{\mathrm{cl}_{\mathsf{S}}(n)}=\sum_{L\in\mathrm{Cl}(\mathsf{S})}a_{h,n,L} \widehat{L}.\]
Since $hn\in \mathrm{cl}_{\mathsf{S}}(h) \mathrm{cl}_{\mathsf{S}}(n)$, it follows that $a_{h,n,\mathrm{cl}_{\mathsf{S}}(hn)}\neq0$. This means that every element $k$ of $\mathrm{cl}_{\mathsf{S}}(hn)$ has the form $k_hk_n$ for some $k_h\in \mathrm{cl}_{\mathsf{S}}(h)$ and $k_n\in \mathrm{cl}_{\mathsf{S}}(n)$. Therefore, for each $k\in\ \mathrm{cl}_{\mathsf{S}}(hn)$ we have 
\[(hn)^{-1}k=n^{-1}h^{-1}k_hk_n=(n^{-1}h^{-1}n)(n^{-1}k_hn)(n^{-1}k_n).\]
Since $H$ is $\mathsf{S}$-normal, it follows that $\mathrm{cl}_{\mathsf{S}}(n^{-1}hn)=\mathrm{cl}_{\mathsf{S}}(h)=\mathrm{cl}_{\mathsf{S}}(k_h)=\mathrm{cl}_{\mathsf{S}}(n^{-1}k_hn)$.
Therefore, we have 
\[(hn)^{-1}k\in [H,\mathsf{S}][N,\mathsf{S}]\]
and the result follows.
\end{proof}
\begin{proof}[Proof of Proposition~\ref{firstfitting}]
By Lemma~\ref{comsplit}, we have
\[\gamma_2(\mathsf{S}_{HN})=[HN,\mathsf{S}_{HN}]=[HN,\mathsf{S}]=[H,\mathsf{S}][N,\mathsf{S}]=\gamma_2(\mathsf{S}_H)\gamma_2(\mathsf{S}_N),\]
so it follows that
\[\gamma_3(\mathsf{S}_{HN})=[\gamma_2(\mathsf{S}_H)\gamma_2(\mathsf{S}_N),\mathsf{S}]=[\gamma_2(\mathsf{S}_H),\mathsf{S}][\gamma_2(\mathsf{S}_N),\mathsf{S}]=\gamma_3(\mathsf{S}_H)\gamma_3(\mathsf{S}_N).\]
Continuing by induction on $i$, we have that
\[\gamma_i(\mathsf{S}_{HN})=\gamma_i(\mathsf{S}_H)\gamma_i(\mathsf{S}_N)\]
for each $i\ge 1$. Since $H$ is $\mathsf{S}_H$-nilpotent, and $N$ is $\mathsf{S}_N$-nilpotent, $\gamma_i(\mathsf{S}_H)\gamma_i(\mathsf{S}_N)$ is trivial for sufficiently large $i$, as required.
\end{proof}

We remark that in the case that $\mathsf{S}=\mathsf{m}(G)$ and $N\lhd G$, then the property of $N$ being $\mathsf{S}_N$-nilpotent is stronger than the property of $N$ just being nilpotent. Indeed, $N$ is $\mathsf{S}_N$-nilpotent means that $[N,G,\dotsc,G]$ is eventually trivial. So although Proposition~\ref{firstfitting} is reminiscent of Fitting's Theorem, it is not an analog.

\begin{lem}\label{quotientcom}If $N$ is $\mathsf{S}$-normal, then $\gamma_{i}(\mathsf{S}_{G/N})=\gamma_i(\mathsf{S})N/N$ for each $i\ge 1$.\end{lem}
\begin{proof}
Let $\pi:G\to G/N$ be the canonical projection. Let $g\in G$ and note that $\mathrm{cl}_{\mathsf{S}_{G/N}}(gN)=\pi(\mathrm{cl}_{\mathsf{S}}(g))$. Thus for any $M\lhd_{\mathsf{S}}G$ such that $N\le M$, we have
\begin{align*}
[M/N,\mathsf{S}_{G/N}]&=\inner{(g^{-1}N)(kN):gN\in M/N,\ \,kN\in \mathrm{cl}_{\mathsf{S}_{G/N}}(gN)}\\
&=\inner{\pi(g^{-1}k):g\in M,\ \,k\in \mathrm{cl}_{\mathsf{S}}(g)}\\
&=\pi([M,\mathsf{S}])=[M,\mathsf{S}]N/N.
\end{align*}
Proceeding by induction on $i$, we have for any $i\ge 1$
\begin{align*}\gamma_{i}(\mathsf{S}_{G/N})&=[\gamma_{i-1}(\mathsf{S}_{G/N}),\mathsf{S}_{G/N}]\\
&=[\gamma_{i-1}(\mathsf{S})N/N,\mathsf{S}_{G/N}]\\
&=[\gamma_{i-1}(\mathsf{S})N,\mathsf{S}]N/N\\
&=[\gamma_{i-1}(\mathsf{S}),\mathsf{S}][N,\mathsf{S}]N/N\ \text{(by Lemma~\ref{comsplit})}\\
&=[\gamma_{i-1}(\mathsf{S}),\mathsf{S}]N/N\\
&=\gamma_i(\mathsf{S})N/N.
\end{align*}
\end{proof}

\begin{prop}\label{baer}
Let $\mathsf{S}\in\mathrm{SCT}(G)$, and let $N$ be $\mathsf{S}$-normal. Then $G$ is $\mathsf{S}$-nilpotent if and only if both $N$ is $\mathsf{S}_N$-nilpotent and $G/N$ is $\mathsf{S}_{G/N}$-nilpotent. 
\end{prop}
\begin{proof}
Since $[N,\mathsf{S}_N]=[N,\mathsf{S}]\le[G,\mathsf{S}]$, it follows that $N$ is $\mathsf{S}_N$-nilpotent. The fact that $G/N$ is $\mathsf{S}_{G/N}$-nilpotent follows from Lemma~\ref{quotientcom}.

Suppose then $N$ is $\mathsf{S}_N$-nilpotent and $G/N$ is $\mathsf{S}_{G/N}$-nilpotent. By Lemma~\ref{quotientcom}, we have $\gamma_i(\mathsf{S})\le N$ for some $i$. So $\gamma_{i+j-1}(\mathsf{S})\le\gamma_{j}(\mathsf{S}_N)$ for each $j\ge 1$. Since $N$ is $\mathsf{S}_N$-nilpotent, $\gamma_{j}(\mathsf{S}_N)$ is eventually trivial, whence so is $\gamma_i(\mathsf{S})$.
\end{proof}

One characterization of finite nilpotent groups is that these are the groups in which every Sylow subgroup is normal. However, this result does not generalize to $\mathsf{S}$-nilpotent groups. Specifically, a Sylow subgroup of the $\mathsf{S}$-nilpotent group $G$ need not be $\mathsf{S}$-normal, as this next example shows.
\begin{exam}
Let $G=C_{pq}$ (the cyclic group of order $pq$), where $p$ and $q$ are distinct primes, and let $x$ be a generator of $G$. Then $P=\inner{x^q}$ is the unique Sylow $p$-subgroup of $G$, and $Q=\inner{x^p}$ is the unique Sylow $q$-subgroup of $G$. It is easy to check that $\mathsf{S}$ is a supercharacter theory of $G$ where
\[\mathrm{Ch}(\mathsf{S})=\bigl\{\mathrm{Ind}_P^G(\lambda):\mathds{1}\neq\lambda\in\mathrm{Irr}(P)\bigr\}\cup\mathrm{Irr}(G/P)\]
and
\[\mathrm{Cl}(\mathsf{S})=\bigl\{\{g\}:g \in P\bigr\}\cup\bigl\{gP:g \in Q\setminus\{1\}\bigr\}.\]
Note that $G>P>1$ is an $\mathsf{S}$-chief series, so every $\mathsf{S}$-chief factor has prime order. Moreover, we have $\mathsf{S}_{G/P}=\mathsf{m}(G/P)$ and $\mathsf{S}_P=\mathsf{m}(P)$, so $\mathsf{S}\in\mathrm{SCT}_{\mathsf{nil}}(G)$ by Proposition~\ref{nilclassification}. However, since every element of $Q$ lies in a different $\mathsf{S}$-class, $Q$ is not $\mathsf{S}$-normal.
\end{exam}
\begin{rem}
The supercharacter theory $\mathsf{S}$ is an example of a $\ast$-{\it product} (see \cite{AH12} for details). Specifically, $\mathsf{S}=\mathsf{m}(P)\ast\mathsf{m}(G/P)$.
\end{rem}
Since the Sylow subgroups of the $\mathsf{S}$-nilpotent group $G$ are not necessarily $\mathsf{S}$-normal, one cannot detect $\mathsf{S}$-nilpotence from the $\mathsf{S}$-table as easily as in the case $\mathsf{S}=\mathsf{m}(G)$. However  $\mathsf{S}$-nilpotence can still be detected from the $\mathsf{S}$-character table.
\begin{prop}\label{nilchartable}
The lower $\mathsf{S}$-central series can be found from the $\mathsf{S}$-character table.
\end{prop}
\begin{proof}
Note that a character $\chi$ is a multiple of a linear character if and only if $\mathbf{Z}(\chi)=G$. Therefore, even if the supercharacters are scaled, one may determine $[G,\mathsf{S}]$ from the $\mathsf{S}$-character table, since
\[[G,\mathsf{S}]=\bigcap_{\substack{\chi\in\mathrm{Ch}(\mathsf{S})\\\chi(1)=1}}\ker(\chi)=\bigcap_{\substack{\chi\in\mathrm{Ch}(\mathsf{S})\\\mathbf{Z}(\chi)=G}}\ker(\chi).\]
For any $i\ge 2$, one obtains the restricted theory $\mathsf{S}_{\gamma_{i}(\mathsf{S})}$ by restricting the $\mathsf{S}$-characters and one may determine which restricted $\mathsf{S}$-characters are multiples of linear characters, we actually have for each $i\ge 2$ that
\[\gamma_i(\mathsf{S})=\bigcap_{\psi\in X_i}\ker(\chi),\]
where
\[X_i=\bigl\{\mathrm{Res}^G_{\gamma_{i-1}(\mathsf{S})}(\chi):\chi\in\mathrm{Ch}(\mathsf{S})\  \,\text{and}\ \,\mathbf{Z}(\mathrm{Res}^G_{\gamma_{i-1}(\mathsf{S})}(\chi))=\gamma_{i-1}(\mathsf{S})\bigr\},\]
and this can be read from the $\mathsf{S}$-character table.
\end{proof}
Although $\mathsf{S}$-character degrees and $\mathsf{S}$-class sizes do not typically divide $\norm{G}$, this does happen in the case that $G$ is $\mathsf{S}$-nilpotent.

\begin{prop}\label{snildiv}
Let $G$ be $\mathsf{S}$-nilpotent. Then $\chi(1)$ divides $\norm{G}$ for each $\chi\in\mathrm{Ch}(\mathsf{S})$, and $\norm{\mathrm{cl}_{\mathsf{S}}(g)}$ divides $\norm{G}$ for each $g\in G$.
\end{prop}
\begin{proof}
Assume that the result does not hold in general; let $G$ be a group of minimal order for which the result fails, and let $\mathsf{S}\in\mathrm{SCT}_{\mathsf{nil}}(G)$ witness this failure. 

Since $G$ is $\mathsf{S}$-nilpotent, it follows that $H=[G,\mathsf{S}]<\norm{G}$. Note that $H$ must be nontrivial by the choice of $G$. Let $\chi\in\mathrm{Ch}(\mathsf{S})$, and let $\psi\in\mathrm{Ch}(\mathsf{S}_H)$ be such that $\inner{\psi,\mathrm{Res}^G_H(\chi)}>0$. Since $H$ is $\mathsf{S}_H$-nilpotent, $\psi(1)$ must divide $\norm{H}$ by minimality. Moreover, $\chi(1)/\psi(1)$ divides $\norm{G:H}$ by Proposition~\ref{actiondivide}. But this means that $\chi(1)$ divides $\norm{G}$.

Also $Z=\mathbf{Z}(\mathsf{S})>1$ since $G$ is $\mathsf{S}$-nilpotent, and $Z<G$ by the choice of $G$. Let $g\in G$, and note that $\norm{\mathrm{cl}_{\mathsf{S}}(g)}/\norm{\mathrm{cl}_{\mathsf{S}_{G/Z}}(gZ)}$ divides $\norm{Z}$ by Proposition~\ref{otherclassdivide}. Now, since $G/Z$ is $\mathsf{S}_{G/Z}$-nilpotent, it follows that $/\norm{\mathrm{cl}_{\mathsf{S}_{G/Z}}(gZ)}$ divides $\norm{G/Z}$ by the minimality of $G$. Hence $/\norm{\mathrm{cl}_{\mathsf{S}}(g)}$ divides $\norm{G}$ as well. But this violates the choice of $G$. Hence no such $G$ exists.
\end{proof}

As a corollary to Proposition~\ref{snildiv}, we have Theorem~\ref{cc}.
\begin{proof}[Proof of Theorem~\ref{cc}]
By Proposition~\ref{snildiv}, we need only show that $\chi(1)$ divides $\norm{G:\ker(\chi)}$ for each $\chi\in\mathrm{Ch}(\mathsf{S})$. To that end, let $\chi\in\mathrm{Ch}(\mathsf{S})$. Since $G$ is $\mathsf{S}$-nilpotent and $\ker(\chi)\lhd_{\mathsf{S}}G$, we have that $G/\ker(\chi)$ is $\mathsf{S}_{G/\ker(\chi)}$-nilpotent by Proposition~\ref{baer}. Since $\chi$ may be considered an element of $\mathrm{Ch}(\mathsf{S}_{G/\ker(\chi)})$, the result follows by Proposition~\ref{snildiv}.
\end{proof}
\begin{rem}\label{snildivrem}
The contrapositive of Theorem~\ref{cc} could be used to determine if $G$ is $\mathsf{S}$-nilpotent from its $\mathsf{S}$-character table. This also implies that if $G$ is $\mathsf{S}$-nilpotent and $\mathsf{T}\preccurlyeq\mathsf{S}$, then $\chi(1)$ and $\norm{K}$ are divisors of $\norm{G}$ for all $\chi\in\mathrm{Ch}(\mathsf{T})$ and $K\in\mathrm{Cl}(\mathsf{T})$.
\end{rem}

If $G$ is nilpotent, then it is straightforward to show that $\chi(1)^2$ divides $\norm{G:\ker(\chi)}$ for each $\chi\in\mathrm{Irr}(G)$. Remarkably, the converse of this statement also holds, as was proved by Gagola and Lewis in \cite{GL99}.  All computational evidence seems to suggest the following generalization.
\begin{conj}\label{kerconj}
Let $\mathsf{S}$ be a supercharacter theory of $G$ and suppose that $\chi(1)$ divides $\norm{G:\ker(\chi)}$ for all $\chi\in\mathrm{Ch}(\mathsf{S})$. Then $G$ is $\mathsf{S}$-nilpotent. In particular, $G$ is nilpotent.
\end{conj}
We remark that the condition $\chi(1)^2$ divides $\norm{G}$ for all $\chi\in\mathrm{Irr}(G)$ is not enough to obtain the Gagola--Lewis result. However, it has been proved (using the classification theorem) that $G$ cannot be simple in this case --- in fact $\mathbf{F}(G)>1$ \cite{SGJr05}. 

Although we do not know if Conjecture~\ref{simple} holds in general, it does hold in the case that $G$ is a $p$-group.

\begin{proof}[Proof of Theorem~\ref{intropgp}]
We have already seen in Theorem~\ref{cc} that {\bf(1)} implies both {\bf(2)} and {\bf(3)}. 

We will now show that {\bf(2)} implies {\bf(1)}. By Theorem~\ref{degreediv}, it suffices to show that $[G,\mathsf{S}]<G$. Since $p\divides\chi(1)$ for each $\chi\in\mathrm{Ch}(\mathsf{S})$ and
\[\norm{G}=\sum_{\chi\in\mathrm{Ch}(\mathsf{S})}\chi(1)=\norm{G:[G,\mathsf{S}]}+\sum_{\chi\in\mathrm{Ch}(\mathsf{S}\mid[G,\mathsf{S}])}\chi(1),\]
it follows that $p$ divides $\norm{G:[G,\mathsf{S}]}$. In particular, we have $[G,\mathsf{S}]<G$. 


Next we show that {\bf(3)} implies {\bf(1)}, which will complete the proof. By Theorem~\ref{otherclassdivide}, it suffices to show that 
$\mathbf{Z}(\mathsf{S})>1$. Since $p\divides\norm{K}$ for each $K\in\mathrm{Cl}(\mathsf{S})$, $p$ must divide $\norm{\mathbf{Z}(\mathsf{S})}$ since
\[\norm{G}=\sum_{K\in\mathrm{Cl}(\mathsf{S})}\norm{K}=\norm{\mathbf{Z}(\mathsf{S})}+\sum_{\substack{K\in\mathrm{Cl}(\mathsf{S})\\\norm{K}>1}}\norm{K}.\]
Therefore $\mathbf{Z}(\mathsf{S})$ is nontrivial. 

\end{proof}

We obtain the following corollary that guarantees the existence of an $\mathsf{S}$-normal version of the $p$-core $\mathbf{O}_p(G)$.
\begin{cor}
Let $\mathsf{S}$ be a supercharacter theory of $G$, and let $p$ be a prime divisor of $\norm{G}$. Then there exists a unique largest $\mathsf{S}$-normal $p$-subgroup $P$ of $G$ for which $P$ is $\mathsf{S}_P$-nilpotent. In particular, this is the unique largest $\mathsf{S}$-normal $p$-subgroup $P$ for which every $\mathsf{S}_P$-class is a $p$-power.
\end{cor}
\begin{rem}\label{pgrpscts}
In Remark~\ref{snildivrem}, we observed that if $G$ is $\mathsf{S}$-nilpotent and $\mathsf{T}\preccurlyeq\mathsf{S}$, then $\chi(1)$ and $\norm{K}$ are divisors of $\norm{G}$ for each $\chi\in\mathrm{Ch}(\mathsf{T})$ and $K\in\mathrm{Cl}(\mathsf{T})$. For $p$-groups, we get a similar but stronger result as a consequence of Theorem~\ref{intropgp}. Let $G$ is a $p$-group, let $\mathsf{S}\in\mathrm{SCT}(G)$ satisfy {\bf(2)} (equivalently {\bf(3)}) of Theorem~\ref{intropgp}, and let $\mathsf{T}\preccurlyeq\mathsf{S}$. Then $G$ is $\mathsf{S}$-nilpotent; hence $G$ is also $\mathsf{T}$-nilpotent. So $\mathsf{T}$ satisfies {\bf(2)} and {\bf(3)} of Theorem~\ref{intropgp}. In particular, this means that if $\chi(1)$ is a divisor of $\norm{G}$ for each $\chi\in\mathrm{Ch}(\mathsf{S})$ and $\mathsf{T}\preccurlyeq \mathsf{S}$, then $\chi(1)$ is a divisor of $\norm{G}$ for each $\chi\in\mathrm{Ch}(\mathsf{T})$. The same statement also holds for superclass sizes. Therefore, to construct all supercharacter theories of a $p$-group efficiently using a computer program, it could be advantageous to find coarse supercharacter theories $\mathsf{S}$ for which $G$ is $\mathsf{S}$-nilpotent.
\end{rem}

We conclude this section with another consequence to and Proposition~\ref{baer} and Theorem~\ref{intropgp}; this is a purely group theoretic fact that we may prove easily using our results about $\mathsf{S}$-nilpotence. This fact also follows immediately from a theorem of Baer, so we present a group theoretic proof afterward.

\begin{prop}
Let $G$ be a finite group, and $N\lhd G$. Assume that $N$ and $G/N$ are nilpotent. Additionally, assume that for every prime divisor $p$ of $G$, and for every element $g\in N$ of $p$-power order, $\mathrm{cl}_G(g)$ has $p$-power size. Then $G$ is nilpotent. 
\end{prop}
\begin{proof}[Proof via supercharacter theory]
To prove this, it suffices to show that $N$ is $\mathsf{m}_G(N)$-nilpotent. Let $p_i$, $1\le i\le \ell$ be the distinct prime divisors of $\norm{N}$, and let $P_i$ be the unique Sylow $p_i$-subgroup of $N$ for each $i$. Since $P_i\Char N$ for each $i$, it follows that $P_i\lhd G$ for each $i$. In particular, this means that each $P_i$ is $\mathsf{m}_G(N)$-normal. Since $\mathrm{cl}_G(g)$ has $p_i$-power size for each $i$, and for each element $g\in N$ of $p_i$-power size, it follows that $P_i$ is $\mathsf{m}_G(P_i)$-nilpotent for each $i$ by Theorem~\ref{intropgp}. Hence $N=\prod_i P_i$ is $\mathsf{m}_G(N)$-nilpotent, by Proposition~\ref{firstfitting}.
\end{proof}

\begin{proof}[Proof via group theory]
Let $\zeta_\infty(G)$ denote the hypercenter of $G$. By a result of Baer \cite[Theorem 3]{RB53}, the condition that all $p$-elements of $N$ belong to $G$-classes of $p$-power size is equivalent to $N$ containing $\zeta_\infty(G)$. This means that $G/\zeta_\infty(G)$ is a quotient of $G/N$, so must be nilpotent as $G/N$ is nilpotent. It follows that $\zeta_\infty(G)=G$, and $G$ is nilpotent.
\end{proof}

\section{Algebra groups}
We will now shift our focus to a particular subset of $p$-groups called algebra groups. Let $A$ be a finite-dimensional, unital, associative $\mathbb{F}$-algebra and let $p=\Char(\mathbb{F})$. Set $J={\bf J}(A)$, the Jacobson radical of $A$. Then $J$ is a nilpotent $\mathbb{F}$-algebra, and $G=\{1+x:x\in J\}$ is a subgroup of the group of units of $A$. A group that arises in this way is called an $\mathbb{F}$-algebra group. In this section, we assume that $\mathbb{F}$ is the field $\mathbb{F}_q$, where $q$ is a power of the prime $p$, and write algebra group to mean $\mathbb{F}_q$-algebra group. In \cite{ID07}, Diaconis and Isaacs construct a supercharacter theory $\mathsf{S}$ for any algebra group $G$ and it turns out that this supercharacter theory always lies in $\mathrm{SCT}_{\mathsf{nil}}(G)$. In this section, we will find a ring-theoretic description of the lower and upper $\mathsf{S}$-central series.

The group $G$ acts on $J$ on the left and the right, and these compatible actions give rise to the two-sided orbits
\[GrG=\{xry:x,y\in G\},\ r\in J\]
which partition $J$. So the collection of sets 
\[K_g=\{1+x(g-1)y:x,y\in G\}\ \,\text{for $g\in G$}\]
partition $G$. It is proven in \cite{ID07} that these sets give the superclasses of a supercharacter theory of $G$. We will call this the {\bf double orbit} supercharacter theory of $G$. 

The supercharacters of the double orbit supercharacter theory are defined dually. Recall that the dual space $J^\ast$ of $J$  is the space of linear functionals $J\to\mathbb{F}_q$. Since $G$ acts on $J$, $G$ also acts on $J^\ast$ from the left and the right by $(g.\mu)(x)=\mu(xg)$ and $(\mu.g)(x)=\mu(gx)$ for every $g\in G$ and $\mu\in J^\ast$. Fix a nontrivial homomorphism $\vartheta:\mathbb{F}_q^+\to\mathbb{C}^\times$. The supercharacter $\chi_\lambda$ associated to the two sided orbit $G\lambda G$ of $\lambda\in J^\ast$ is defined by
\[\chi_\lambda(g)=\frac{\norm{G\lambda}}{\norm{G\lambda G}}\sum_{\mu\in G\lambda G}(\vartheta\circ\mu)(g-1)\]
for each $g\in G$. However, since we have
\[\chi_\lambda=\frac{\norm{G\lambda}}{\norm{G\lambda G}}\sigma_{\mathrm{Irr}(\chi_\lambda)}\]
by \cite[Equation 3.8]{CA08}, we instead define the supercharacters of this theory by
\[\{\norm{G\lambda G}/\norm{G\lambda}\chi_\lambda:\lambda\in J^\ast\}\]
to be consistent with our definition of supercharacter.

Before discussing the central series associated to the double orbit supercharacter theory of an algebra group $G$, we need the following result. We remark that statement {\bf(2)} below is due to Isaacs \cite[Lemma 3.2]{MI95}. We include the proof though, as it sheds light on what the isomorphism is.

\begin{lem}\label{algebragps}Let $J$ be a finite dimensional, associative, nilpotent $\mathbb{F}_q$-algebra. Let $I\subseteq J$ be an ideal of $J$ that is an $\mathbb{F}_q$-subspace. If $G=1+J$, and $\mathsf{S}$ is the double orbit supercharacter theory of $G$, then 
\begin{enumerate}[label = {\bf(\arabic*)}]\openup 5pt
\item the subgroup $N=1+I$ is an algebra group and $N\lhd_{\mathsf{S}}G$;
\item the quotient group $G/N$ is an algebra group, and $G/N$ is naturally isomorphic to $1+J/I$;
\item the double orbit supercharacter theory of $G/N$ is $\mathsf{S}_{G/N}$.
\end{enumerate}
\end{lem}
\begin{proof}
By the definition of $N$, $N$ is an algebra group. Since for any $g\in N$, and $x,y\in G$, we have
\begin{align*}
x(g-1)y&=(1+x-1)(g-1)(1+y-1)\\
&=(g-1)[1+(y-1)+(x-1)+(x-1)(y-1)]\in I,
\end{align*}
it follows that $N\lhd_{\mathsf{S}}G$.

Since $I$ is an ideal of $J$ and is an $\mathbb{F}_q$-subspace, it follows that $I$ is an ideal of $A=\mathbb{F}_q+J$. Therefore $J/I=\mathbf{J}(A/I)$, so $1+J/I$ is an algebra group. Let $\upsilon:A\to A/I$ be the canonical projection. Then $\upsilon\vert_G$ is a group homomorphism with kernel $1+I=N$, which proves {\bf(2)}.

Write $\pi=\upsilon\vert_G:G\to G/N$. Then every $\mathsf{S}_{G/N}$-class has the form $\pi(K_g)$ for some $g\in G$. Since $\pi$ is a homomorphism, it follows that $\pi(K_g)$ is the double orbit superclass of $(g-1)I$, and statement {\bf(3)} holds. 
\end{proof}

Let $J$ be a finite-dimensional, associative, nilpotent algebra. For each $i\ge0$, $J^{i-1}/J^i\le\mathrm{Ann}(J/J^i)$, where $\mathrm{Ann}(I)$ denotes the two-sided annihilator of $I$. Moreover, since $J$ is nilpotent, $J^c=0$ for some integer $c$. A series of ideals $J\supseteq I_1\supseteq I_2\supseteq\dotsb\supseteq I_r=0$ satisfying $I_i/I_{i+1}\subseteq\mathrm{Ann}(J/I_{i+1})$ is known as an annihilator series of $J$. The series $J\supseteq J^2\supseteq J^3\supseteq\dotsb$ is called the lower annihilator series of $J$. Also $J$ is nilpotent if and only if $J^k=0$ for some integer $c$.

We define another sequence of ideals $\mathrm{Ann}_i(J)$ by $\mathrm{Ann}_0(J)=0$, and \[\mathrm{Ann}_{i+1}(J)/\mathrm{Ann}_i(J)=\mathrm{Ann}(J/\mathrm{Ann}_i(J)).\] This is another example of an annihilator series, called the upper annihilator series of $J$. Let $c$ be the maximal integer for which $J^c>0$. Since $J^c\subseteq\mathrm{Ann}(J)$, it follows that $J^c\mathrm{Ann}(J)\subseteq\mathrm{Ann}(J)$. Therefore we have 
\[J^{c-1}\mathrm{Ann}(J)/\mathrm{Ann}(J)\subseteq\mathrm{Ann}(J/\mathrm{Ann}(J))=\mathrm{Ann}_2(J)/\mathrm{Ann}(J),\]
and it follows that $J^{c-1}\le\mathrm{Ann}_2(J)$. Continuing this argument inductively, we have $J^{c+1-i}\subseteq\mathrm{Ann}_i(J)$, and so $J\subseteq\mathrm{Ann}_c(J)$, and the upper annihilator series terminates in $J$ in at most $c$ steps. 

If $r$ is the length of the upper annihilator series, then $J^2=J\mathrm{Ann}_r(J)\subset\mathrm{Ann}_{r-1}(J)$, since $J/\mathrm{Ann}_{r-1}(J)=\mathrm{Ann}(J/\mathrm{Ann}_{r-1}(J))$. Therefore, we must have $J^2\subseteq\mathrm{Ann}_{r-1}(J)$. Again, this inductive argument implies that $J^i\subseteq\mathrm{Ann}_{r-i+1}(J)$ for each $i\ge1$. Then $J^{r+1}=0$, so we must have $r=c$. In particular, this means that $J$ is nilpotent if and only if the upper annihilator series terminates in $J$, and this happens in the same number of steps as it takes $J^i$ to get to $0$.

These arguments are nearly identical to the ones for groups with central series. The analogy here is that group commutation is replaced by ring product and centralization is replaced by annihilation. Although the lower and upper central series of an algebra group do not match up perfectly with this analogy, they do if the notion of center and commutator subgroup are replaced with their supercharacter theoretic versions. 

\begin{thm}\label{upperloweralgebra}
Let $G=1+J$ be an algebra group, and let $\mathsf{S}$ be the double orbit supercharacter theory of $G$. Then the following hold.
\begin{enumerate}[label = {\bf(\arabic*)}]\openup 5pt
\item For each $i\ge1$, $\gamma_i(\mathsf{S})=1+J^i$.
\item For each $i\ge0$, $\zeta_i(\mathsf{S})= 1+\mathrm{Ann}_i(J)$.
\end{enumerate}
\end{thm}
\begin{proof}
First, we prove {\bf(1)} by induction on $i$. Since $1+J=G=\gamma_1(\mathsf{S})$, the base case holds. Now assume that $\gamma_i(\mathsf{S})=1+J^i$ for some $i\ge 1$. Note that $\gamma_{i+1}(\mathsf{S})$ is generated by the set $1+\mathcal{G}_i$, where
\begin{align*}\mathcal{G}_i&=\{g^{-1}-1+g^{-1}x(g-1)y:g\in\gamma_{i+1}(\mathsf{S}),\ x,y\in G\}\\
&=\{g-1-x(g-1)y:g\in\gamma_{i+1}(\mathsf{S}),\ x,y\in G\}\\
&=\{-b+(1+a)b(1+c):b\in J^i,\ b,c\in J\}\\
&=\{ab+ac+abc:b\in J^i,\ a,c\in J\}\subseteq J^{i+1}.
\end{align*}
Therefore, we have $\gamma_{i+1}(\mathsf{S})\le 1+J^{i+1}$.  Since $\gamma_i(\mathsf{S})$ is $\mathsf{S}$-normal, \cite[Proposition 3.2]{EM11} guarantees that the set $\mathcal{H}=\{g-1:g\in\gamma_i(\mathsf{S})\}$ is an $\mathbb{F}_p$-algebra. Therefore, since $J^{i+1}$ is generated by elements of the form $ab$, where $a\in J$ and $b\in J^i$, and the set of these elements is contained in $\mathcal{G}_i\subseteq\mathcal{H}$, we must also have $1+J^{i+1}\le\gamma_{i+1}(\mathsf{S})$.

We next prove {\bf(2)} by induction on $i$. Again, the base case is trivial, so assume that for some $i\ge0$, $\zeta_i(\mathsf{S})= 1+\mathrm{Ann}_i(J)$. Write $Z=\zeta_i(\mathsf{S})$. Note that $G/Z$ is naturally isomorphic to $1+J/\mathrm{Ann}_i(J)$, and $\mathsf{S}_{G/Z}$ is the double orbit supercharcter theory of $G/Z$, by Lemma~\ref{algebragps}. Let $\varphi:G/Z\to1+J/\mathrm{Ann}_i(J)$ be the isomorphism $gZ\mapsto 1+(g-1)\mathrm{Ann}_i(J)$. Then $\varphi(\mathbf{Z}(\mathsf{S}_{G/Z}))$ is exactly the set
\[
1+\{g\in J/\mathrm{Ann}_i(J): (1+a)g(1+b)=g\ \,\text{for all $a,b\in J/\mathrm{Ann}_i(J)$}\},
\]
which is $1+\mathrm{Ann}(J/\mathrm{Ann}_i(J))$.
Therefore \[\zeta_{i+1}(\mathsf{S})/\zeta_i(\mathsf{S})=\varphi^{-1}\bigl(1+\mathrm{Ann}(J/\mathrm{Ann}_i(J))\bigr)=(1+\mathrm{Ann}_{i+1}(J))/(1+\mathrm{Ann}_i(J)),
\]
and it follows that $\zeta_{i+1}(\mathsf{S})=1+\mathrm{Ann}_{i+1}(J)$, as desired.
\end{proof}
The following is an immediate corollary to Theorem~\ref{upperloweralgebra}.
\begin{cor}\label{doubleorbitsize}The degrees of all of the supercharacters and the sizes of all of the superclasses of the double orbit supercharacter theory of an $\mathbb{F}_q$-algebra are powers of the prime $p$, where $p$ is the characteristic of $p$. 
\end{cor}
We remark that a statement stronger than Corollary~\ref{doubleorbitsize} is already known to be true. Specifically, if $\chi=\norm{G\lambda G}/\norm{G\lambda}\chi_\lambda$, then $\chi(1)=\norm{G\lambda G}$, which is a $q$-power by \cite[Lemma 4.3]{ID07}. Also the size of each superclass is a $q$-power, by \cite[Corollary 3.2]{ID07}.

If $G=1+J$ is an algebra group, an algebra subgroup of $G$ is a subgroup of the form $1+I$, where $I$ is a multiplicatively closed $\mathbb{F}_q$-subspace of $J$. In case $I$ is an ideal of $A=\mathbb{F}_q+J$, $1+I$ is called an ideal subgroup of $G$. In \cite[Proposition 3.2]{EM11}, Marberg shows that any $\mathsf{S}$-normal subgroup of $G$ has the form $1+M$, where $M$ is a multiplicatively closed $\mathbb{F}_p$-subspace of $J$ and $\mathsf{S}$ is the double orbit supercharacter theory of $G$. Since every $\mathsf{S}$-chief factor has order $p$, it follows that $G$ has an $\mathsf{S}$-normal subgroup $H$ of order $d$ with this form for every divisor $d$ of $\norm{G}$. It turns out that there are ideal subgroups among these for every $q$-power dividing $\norm{G}$. This is Theorem~\ref{idealsubs}. 

\begin{proof}[Proof of Theorem~\ref{idealsubs}]
Let $1\le m\le n$. Since $\gamma_0(\mathsf{S})=1$ and $\gamma_c(\mathsf{S})=G$ for some integer $c$, we may choose $i$ such that $\norm{\zeta_i(\mathsf{S})}\le q^m < \norm{\zeta_{i+1}(\mathsf{S})}$. If $q^m=\norm{\zeta_i(\mathsf{S})}$, then we are done, so assume that $\norm{\zeta_i(\mathsf{S})}< q^m < \norm{\zeta_{i+1}(\mathsf{S})}$, and recall that $\zeta_j(\mathsf{S})=1+\mathrm{Ann}_j(J)$ for all $j$, and $\zeta_{i+1}(\mathsf{S})/\zeta_i(\mathsf{S})\simeq \mathrm{Ann}_{i+1}(J)/\mathrm{Ann}_i(J)$. Let $M/\mathrm{Ann}_i(J)$ be an $\mathbb{F}_q$-subspace of $\mathrm{Ann}_{i+1}(J)/\mathrm{Ann}_i(J)$ of order $q^m$. Since $\mathrm{Ann}_{i+1}(J)/\mathrm{Ann}_i(J)\subseteq\mathrm{Ann}(J/\mathrm{Ann}_i(J))$, it follows that $M/\mathrm{Ann}_i(J)$ is an ideal of $J/\mathrm{Ann}_i(J)$. Therefore, $M$ is an ideal of $J$ that is an $\mathbb{F}_q$-subspace, so $1+M\lhd_{\mathsf{S}}G$ by Lemma~\ref{algebragps}.
\end{proof}
\bibliographystyle{alpha}
\bibliography{bio}
\end{document}